\documentclass[reqno, centertags, 11pt]{amsart}
\usepackage{geometry, enumerate, enumerate}                
\usepackage{graphicx}
\usepackage{amssymb, color}
\usepackage{epstopdf, soul}

\newcommand{\al}{^{\alpha}}
\newcommand{\x}{{\cdot}}
\newcommand{\cmp}{^{\ast}}

\newcommand{\mb}[1]{\mathbf{ #1}}
\newcommand{\mc}[1]{\mathcal{ #1}}

\newcommand{\Con}{\mathrm{Con}(\mathbf A)}

\newcommand{\te}{\theta}

\newcommand{\ida}{\mathrm{Id}(\mathbf{A})}
\newcommand{\idb}{\mathrm{Id}(\mathbf{B})}
\newcommand{\Lns}{\mb A=\langle A, +,\x,\al,0,1\rangle}
\newcommand{\Lnsb}{\mb B=\langle B, +,\x,\al,0,1\rangle}
\newcommand{\skel}{\mathrm{Skel}(\mathrm{Id}(\mathbf{A}))}

\newcommand{\alga}{\mathbf A}
\newcommand{\algb}{\mathbf B}

\providecommand{\U}[1]{\protect\rule{.1in}{.1in}}
\newtheorem{theorem}{Theorem}
\theoremstyle{plain}

\newtheorem{corollary}{Corollary}

\newtheorem{definition}{Definition}

\newtheorem{lemma}{Lemma}

\newtheorem{proposition}{Proposition}

\numberwithin{equation}{section}

\title{On the structure theory of \L ukasiewicz Near Semirings}
\author{Ivan Chajda}
\address{I. Chajda, Palack\'{y} University, Olomouc\\
Czech Republic}
\email{ivan.chajda@upol.cz}

\author{Davide Fazio}
\address{D. Fazio, Universit\`a di Cagliari, Cagliari\\
Italy}
\email{dav.faz@hotmail.it}

\author{Antonio Ledda}
\address{A. Ledda, Universit\`a di Cagliari, Cagliari\\
Italy}
\email{antonio.ledda@unica.it}

\date{\today}                                           
\thanks{{\bf Corresponding Author}: Antonio Ledda ({\ttfamily antonio.ledda@unica.it})}
\begin{document}
\maketitle
\begin{abstract}
 In a previous article by two of the present authors and S. Bonzio, \L ukasiewicz near semirings were introduced and it was proven that basic algebras can be represented (precisely, are term equivalent to) as near semirings. In the same work it has been shown that the variety of \L ukasiewicz near semirings is congruence regular. In other words, every congruence is uniquely determined by its $0$-coset. Thus, it seems natural to wonder wether it could be possible to provide a set-theoretical characterization of these cosets. This article addresses this question and shows that  kernels can be neatly described in terms of two simple conditions. As an application, we obtain a concise characterization of ideals in \L ukasiewicz semirings. Finally, we close this article with a rather general Cantor-Bernstein type theorem for the variety of involutive idempotent integral near semirings.
\end{abstract}

\section{Introduction}
The notion of near semiring has been introduced by H. L\"anger and one of the authors in \cite{ChajdaL15, ChajdaL} in order to provide a representation of several prominent algebraic structures arising from the theory of quantum mechanics. Taking up some ideas from Belluce, ~Di~Nola, and Ferraioli  \cite{DiNola13}, this concept has been enriched, in an article by S. Bonzio and two of the present authors, by an antitone involution, and the \L ukasiewicz axiom. We termed these algebras \L ukasiewicz near semirings \cite{BCL}. In the same article we discussed the fact that basic algebras and orthomodular lattices can be represented as \L ukasiewicz near semirings, and, furthermore, the smooth structure theory of these algebras was investigated. Indeed, \L ukasiewicz near semirings are congruence regular; i.e. every congruence is completely determined by its $0$-coset. In this article, this observation leads us to look for a  set-theoretical characterization of kernels; namely, a notion of \emph{ideal} that properly matches with congruences. As an application of this result, we obtain a concise set-theoretical characterization of ideals in \L ukasiewicz semirings. In particular, in case an element $e$ is \emph{central} (i.e. it induces a pair of factor congruences) then the ideal it generates is amenable of a neat order-theoretical characterization: it corresponds to the interval $[0,e]$. Although the notion of centrality is easily captured in the variety of involutive idempotent integral near semiring and \L ukasiewicz near semiring, this concept yields rather strong properties. Indeed, by virtue of this characterization, in the last section of this article, we propose a rather comprehensive algebraic version of the Cantor-Bernstein theorem for the variety of involutive idempotent integral near semirings. It seems to us that this theorem implies a fairly general fact: even if an algebra is not a lattice (involutive idempotent integral near semirings, in general, need not be lattices, but semilattices) its inner structure is captured by means of the intervals $[0,e]$ (which are indeed ideals in \L ukasiewicz near semirings!), with $e$ a central element. This theorem subsumes analogous results for basic algebras, orthomodular lattices, and MV-algebras, since these structures  are term-equivalent subvarieties of the variety of involutive idempotent integral near semirings (see \cite{BCL}).
 
This article is structured as follows: in section \ref{sec:1} we introduce the notion of ideals in \L ukasiewicz near semirings, in section \ref{sec:2} we discuss centrality in the same context, and finally, in section \ref{sec:3}, we deal with a rather general algebraic version of the Cantor-Bernstein theorem for the variety of involutive idempotent integral near semirings.

\section{Ideals in \L ukasiewicz near semirings}\label{sec:1}
We begin this section with the definition of our main concepts.

\begin{definition}\label{def: Lnrsmrng}
An \emph{{involutive idempotent integral} near semiring}\footnote{{Let us remark that in general, cf. \cite{BCL} and \cite{ChajdaL15}, the notion of involutive semiring do not involve integrality. Indeed, there are involutive semirings which are not integral (cf. \cite[Remark 3]{BCL}).}} (briefly, $\iota$-near semiring) is an algebra $ \mathbf{A}=\langle A, +, \cdot , \al,0, 1\rangle $ of type $ \langle 2, 2, 1, 0, 0\rangle $ such that
\begin{itemize}
\item[(i)] the reduct $\langle{A, +,0,1}\rangle$ is a join semilattice with $0,1$ the smallest and the greatest element, respectively, and $\leq$ is the induced order (i.e. $x\leq y$ if and only if $x+y=y$);
\item[(ii)] $ \langle A, \cdot, 1 \rangle $ is a groupoid satisfying $ x\cdot 1 \approx x \approx 1\cdot x $ (a unital groupoid);
\item[(iii)] $ (x+y)\cdot z \approx (x\cdot z) + (y\cdot z) $;
\item[(iv)] $ x\cdot 0 \approx 0\cdot x \approx 0 $;
\item[(v)] ${x\al}\al\approx x$;
\item[(vi)] if $x\leq y$, then $y\al\leq x\al$.
\end{itemize}
A {$\iota$-near semiring} semiring is a \emph{\L ukasiewicz near semiring} if it satisfies the following further condition:

\begin{itemize}
 \item[(vii)]$(x\x y\al)\al\x y\al\approx (y\x x\al)\al\x x\al$.
\end{itemize}
\end{definition}

Let us remark that in any involutive near semiring one has that $0\al=1$. Furthermore, it is easily seen that, since $x\leq x+y$ (by (i)), it follows that $(x+y)\al\leq x\al$ (by (vi)). Hence, we have that 
\begin{equation}
(x+y)\al + x\al =x\al.\tag{viii}
\end{equation}
For notational clarity, whenever it's possible, we will omit the symbol ``$\x$'' and use juxtaposition: by $xy $ we mean $x\x y$.

The following result is Lemma 3 in \cite{BCL},
\begin{proposition}\label{prp: lmm3}
In any \L ukasiewicz near semiring the following identities hold:
\begin{enumerate}[(a)]
 \item$xx\al\approx x\al x\approx0$;
 \item$x+y\approx((x\x y\al)\al\x y\al)\al$.
\end{enumerate}
\end{proposition}
A \emph{\L ukasiewicz semiring} $\mb A$ is a \L ukasiewicz near semiring such that the reduct $\langle A,\x,1\rangle$ is a monoid. It follows from Theorem 2 in \cite{BCL}, that in any \L ukasiewicz semiring the groupoidal operation $\x$ is \emph{commutative} and \emph{right distributive}: the equation $z\x(x+y)=(z\x x)+(z\x y)$ is satisfied. In other words, the reduct $\langle A, +,\x,0,1\rangle$ is a semiring. Since \L ukasiewicz near semirings are congruence regular \cite[Theorem 7]{BCL}, every congruence $\te$ is fully specified by its kernel $[0]_{\te}$. Therefore, it seems quite reasonable to wonder whether this class could be amenable of a smooth set-theoretical characterization. With this aim in mind we introduce the following definition:
\begin{definition}\label{def:idl}
Let $\mb A$ be a \L ukasiewicz near semiring. A set $I\subseteq A$ is called an \emph{ideal} if $0\in I$ and the following conditions hold:
\begin{enumerate}
\item[(I1)] if $ab\al\in I$ and $b\in I$, then $a\in I$;
\item [(I2)] if $a\al b,\,b\al a\in I$, then $(ac)\al\x(bc),\,(ca)\al\x(cb)\in I$, for any $c\in A$.
\end{enumerate}
\end{definition}
Let us observe that, setting $c=b\al$ in condition (I2) we immediately obtain
\begin{enumerate}
 \item[(I3)]  if $a\al b,\,b\al a\in I$, then $ab\al\in I$.
\end{enumerate}
We will denote by $\mathrm{Con}(\alga)$ and $\mathrm{Id}(\alga)$ the sets of congruences and ideals of $\alga$, respectively.\\

Let us observe that, for any congruence $\te$ on a \L ukasiewicz near semiring $\mb A$, and any $a\in A$, $[a]_{\te}$ is convex. In fact, if $c\in [a]_{\te}$ and $a\le b\le c$, then $b=a+b\te c+b=c$.
The following lemma characterizes, for every congruence, the relative kernel. It can be seen that, for any  \L ukasiewicz near semiring $\alga$, the following facts hold true.

\begin{lemma}\label{lem: krnl1}
If $\te\in \mathrm{Con}(\alga)$, then $a\te b$ if and only if $a\al b,\,b\al a\in [0]_{\te}$. 
\end{lemma}
\begin{proof}
If $a\te b$, then $a\al b \te b\al b=0$, and dually for $b\al a$. Conversely, if $a\al b,\,b\al a\in [0]_{\te}$, then if $(a\al b)\al\te0\al=1$, and so $(a\al b)\al b\te1b=b$, and dually $(b\al a)\al a\te a$. But then, $b\te(a\al b)\al b=(b\al a)\al a\te a$.
\end{proof}
It turns out that, for any congruence $\te$, the coset $[0]_{\te}$ is an ideal.
\begin{theorem}\label{thm:cst-ideal}
If $ \te\in \mathrm{Con}(\alga)$, then $[0]_{\te}\in  \mathrm{Id}(\alga)$.
\end{theorem}
\begin{proof}
It is clear that $0\in [0]_{\te}$. For (I1), if $ab\al\in[0]_{\te}$ and $[b]_{\te}=[0]_{\te}$, then $[0]_{\te}=[ab\al]_{\te}=[a]_{\te} [b\al]_{\te}=[a]_{\te} [b]_{\te}\al=[a]_{\te}[0]_{\te}\al=[a]_{\te}[1]_{\te}=[a]_{\te}$. 
Finally, for condition  (I2), if $a\al b,\,b\al a\in [0]_{\te}$, again by Lemma \ref{lem: krnl1}, $a\te b$. Hence, $ac\te bc$ and $ca\te cb$. Therefore, by Lemma \ref{lem: krnl1}, $(ac)\al(bc)\te 0$ and $(ca)\al(cb)\te0$.
\end{proof}
Conversely,
\begin{theorem}\label{thm:idl-cngrnc}
If $I\in \mathrm{Id}(\alga)$, then the relation $\te(I)$, defined for all $a,\,b\in A$ by
\begin{equation}\label{eq:1}
 a\te(I)b\Leftrightarrow a\al b,\,b\al a\in I,
\end{equation}
is a congruence on $\mb A$, and $[0]_{\te(I)}=I$.  
\end{theorem}
\begin{proof}
Reflexivity and symmetry are straightforward. As regards transitivity, suppose that $a\al b,\,b\al a,\,b\al c,c\al b\in I$, then, by condition (I2), $(c\al a)\al(c\al b),\, (c\al b)\al(c\al a)\in I$. So, by condition (I3), $(c\al a)(c\al b)\al\in I$.
Because $c\al b\in I$, from condition (I1), $c\al a\in I$. By assumption, and condition (I3) ${a\al}\al b\al,\,{b\al}\al a\al\in I$. From condition (I2) we obtain that $(a\al c)\al(b\al c)\in I$ and $(b\al c)\al(a\al c)\in I$. By (I3), $(a\al c)(b\al c)\al\in I$. Now, $b\al c\in I$, and so by (I1) $a\al c\in I$. As regards the operations, it is straightforward from (I2) and (I3), respectively, that $\x$ and $\al$ are preserved. From this fact we have that, if $a\te(I)b$, then $a+c=((a\x c\al)\al\x c\al)\al\te(I)((b\x c\al)\al\x c\al)\al=b+c$, by Proposition \ref{lem: krnl1}. Finally, if $a\in I$, then $1a=0\al a\in I$ and $a\al 0=0\in I$, and so $a\in [0]_{\te(I)}$. Conversely, if $a\in [0]_{\te(I)}$, then $a\al0,\,0\al a=1a=a\in I$, which proves that $I=[0]_{\te(I)}$.
\end{proof}
As mentioned above, in any \L ukasiewicz near semiring $\alga$, if the multiplication operation is associative, then it is also commutative. Therefore, $\mb A$ would be a \emph{\L ukasiewicz semiring}. \L ukasiewicz semirings are objects of prominent importance for algebraic logic. In fact, MV-algebras, the equivalent algebraic semantics of \L ukasiewicz many-valued logic, can be represented in terms of \L ukasiewicz semiring.
More specifically, upon setting
\[
x\oplus y=((x\al+y)\x y\al)\al,
\]
the structure $\mc M(\mb A)=\langle A, \oplus,\al,0\rangle$ is an MV-algebra, and, conversely, if $\mb B=\langle A,\oplus,',1\rangle$ is an MV-algebra, upon defining
\[
x+y=(x'\oplus y)'\oplus y,\;x\x y=(x'\oplus y)',\,1=0,'\mbox{ and }x\al=x'
\]
the structure $\mc R(\mb B)=\langle B, +,\x,\al,0,1\rangle$ is a \L ukasiewicz semiring. Moreover, these correspondences are mutually inverse (cfr. Theorem 6 and Corollaries 3,4 in \cite{BCL}). It seems to us that it could be of some interest wondering how the notion of ideal, in the general setting of \L ukasiewicz near semirings, would specify to the case of \L ukasiewicz semirings. Actually, for a \L ukasiewicz semiring $\mb A$, we have that:
\begin{corollary}\label{cor:Lukasiewicz semiring}
A set $I\subseteq A$ such that $0\in I$ is the kernel of some congruence $\te$ if and only if it satisfies conditions (I1) and (I2) of Definition \ref{def:idl}. Moreover, $I=[0]_{\te(I)}$, where $\te(I)$ is as in condition \eqref{eq:1} in Theorem \ref{thm:idl-cngrnc}.
\end{corollary}
Futhermore, one can easily prove that ideals in \L ukasiewicz semirings can be defined in the same way they are defined in the case of commutative semirings.
\begin{proposition}
Let $\Lns$ be a \L ukasiewicz semiring. Then $I\subseteq A$ is an ideal if and only if the following conditions hold:
\begin{enumerate}[(i)]
\item $0\in I$;
\item $a,b\in I$ implies $a+b\in I$; 
\item $a\in I$ implies $a\x c=c\x a\in I$, for any $c\in A$.
\end{enumerate}
\end{proposition}
\begin{proof}
Let $I$ be and ideal in $\alga$. 
We only need to prove that conditions (ii) and (iii) are satisfied. Let $a\in I$. One has that $0=c(aa\al)=(ca)a\al=(ac)a\al\in I$. Hence, by (I1), it follows that $ac\in I$. So condition (iii) holds. Now, assume that $a,b\in I$. By condition (iii), we obtain that $(a+b)a\al =0+ba\al\in I$, so (I1) yields $(a+b)\in I$. Thus, (ii) is proved. Conversely, if (i)-(iii) hold, it is easily seen that (I1) and (I2) are satisfied. In fact, suppose that $ab\al,b\in I$. By (iii) $ab\in I$, hence by (ii) $ab\al + ab=a(b\al +b)=a\x 1\in I$. Finally, assuming that $a\al b,b\al a\in I$, one has by condition (iii) and Definition \ref{def: Lnrsmrng}(viii) that $(ac)\al (bc)= ((ac)\al c)b=((c\al a\al)\al a\al)b=(c\al a\al)\al (a\al b)\in I$.
\end{proof}

Let $\alga$ be a \L ukasiewicz near semiring. Hence, a straightforward verification proves that the structure $\langle\ida ,\wedge ,\vee ,\{0\}, A\rangle$ is a complete lattice under the set-theoretic ordering with operations $I\wedge J=I\cap J$ and $I\vee J=\langle I\cup J\rangle$ (i.e, the least ideal containing both $I$ and $J$). In what follows, we will call this structure the \textit{ideal lattice} of $\alga$. Moreover, the one-to-one correspondence between $\ida$ and $\Con$ stated by Theorems \ref{thm:cst-ideal} and \ref{thm:idl-cngrnc} is, in fact, an isomorphism.

\begin{theorem}\label{thm: isom ida con}
The ideal lattice of $\mb A$ is isomorphic to $\Con$. Hence, $\ida$ is an algebraic and distributive lattice.
\end{theorem}
\begin{proof}
Let $f :\ida\rightarrow\Con$ be the mapping defined by $f  (I)=\theta (I)$. By Theorems \ref{thm:cst-ideal} and \ref{thm:idl-cngrnc}, $f$ is a bijection, and its inverse $g :\Con\rightarrow\ida$ is $g (\theta)=[0]_{\theta}$. Now, it should only be proved that $f$ is an homomorphism. Clearly, $f(I\cap J)=\theta (I)\cap\theta (J) =f(I)\wedge f(J)$. Now we show that $f(I\vee J)=f (\langle I\cup J\rangle )=\theta (\langle I\cup J\rangle)=\theta(I)\lor \theta(J)$. By Lemma \ref{lem: krnl1}, we have that $(a,b)\in\theta (I)\vee\theta (J)$  if and only if $a\al b,b\al a\in [0]_{\theta (I)\vee\theta (J)}$. Note that, by congruence permutability (cf. page \pageref{eq:cng prmtblt}), $a\al b\in [0]_{\theta (I)\vee\theta (J)}$ if and only if there is a $c$ such that %
\[
a\al b\theta (I)c\theta (J)0.
\]
Therefore, again by Lemma \ref{lem: krnl1} and Theorem \ref{thm:idl-cngrnc}: 

\[
(a\al b)\al c, c\al(a\al b)\in I\text{ and }c\in J.
\]

Therefore, by (I3), $(a\al b) c\al\in I$. Then, by condition (I1), we have that $a\al b\in \langle I\cup J\rangle$, and by symmetry $b\al a\in \langle I\cup J\rangle$. For the other inclusion, note that $I, J\subseteq [0]_{\theta (I)\vee\theta (J)}$. Hence $\langle{I\cup J}\rangle\subseteq [0]_{\theta (I)\vee\theta (J)}$. Therefore, by Theorem \ref{thm:idl-cngrnc}, $\theta(\langle{I\cup J}\rangle)\subseteq \theta (I)\vee\theta (J)$.
Then, it turns out that $f (I)\vee f(J)=\theta (I)\vee\theta (J)=\theta (\langle I\cup J\rangle)$. Hence $f$ is an isomorphism. Finally, since $\Con$ is both distributive (see \cite{BCL}) and, of course, algebraic, $\ida$ is a distributive and algebraic lattice.
\end{proof}
It might be useful to emphasize that the result above is, in fact, an explicit proof of a general result due to H.P. Gumm and A. Ursini. In fact, \cite[Corollary 1.9]{GU84} proves that a variety $\mathcal{V}$, equipped with a constant $0$, is \textit{ideal determined} (namely, for any $\alga\in\mathcal{V}$ there is a one to one correspondence between $\Con$ and $\ida$) if and only if $\mathcal{V}$ is 0-regular and there exists a binary term $s(x,y)$ such that 
\[
\mathcal{V}\models s(x,x)=0\ \ \mathrm{and}\ \ \mathcal{V}\models s(0,x)=x.
\]

Thus, since \L ukasiewicz near semirings are congruence regular, putting $s(x,y)=x\al y$, they are an ideal determined variety. Futhermore, it can be easily seen that the previous result provides a rather concise description of the ideals of the form $\langle I\cup J\rangle$ with $I, J\in\ida$. Let
\[
[I]_{\te (J)}=\{ a\in A | (a,i)\in\te (J)\ \mathrm{for\ some}\ i\in I \},
\]
for any $I,J\in\ida$. In any \L ukasiewicz near semiring $\alga$, we can prove:
\begin{proposition}For any $I,J\in\ida$:
\[
\langle I\cup J\rangle =[I]_{\te (J)}.
\]

\end{proposition}
\begin{proof} If $a\in\langle I\cup J\rangle$, then $a\in [0]_{\te(I)\vee\te(J)}$ and there exists $k<\omega$ such that $$a\theta(I)c_{1}\theta(J)c_{2}\theta(I)...c_{k}\theta(J)0.$$ Since \L ukasiewicz near semirings are congruence permutable, one has that there exists $c\in A$ such that $a\theta (J)c\theta(I)0$. Thus, $\langle I\cup J\rangle\subseteq [I]_{\te(J)}$. Conversely, if $a\in [I]_{\te(J)}$, then $a\theta(J)i\theta(I)0$ for some $i\in I$. Hence, $a\in[0]_{\te(J)\vee\te(I)}=\langle I\cup J\rangle$.
\end{proof}
As we have mentioned, $\ida$ is algebraic with $\{0\}$ and $A$ its least and the greatest element, respectively. Hence,  \textit{it has the infinite join distributive property}, see e.g. \cite{Gratzer:2011}. It means that, for any ideal $J\in\ida$, and an arbitrary family of ideals $\{ I_{\gamma}\}_{\gamma\in\Gamma}$, it holds that
\begin{equation}\label{eq: 2}
J\cap\bigvee\{I_{\gamma}|\gamma\in\Gamma\} =\bigvee\{J\cap I_{\gamma}|\gamma\in\Gamma\}.
\end{equation}
From this fact, we can deduce that:
\begin{theorem}\label{thm: ida pseudocompl}
The ideal lattice $\ida$ of any \L ukasiewicz near semiring $\alga$ is pseudocomplemented.
\end{theorem}
\begin{proof}
Let $J\in\ida$ and consider the set $$S^{J}=\{I\in\ida|J\cap I=\{0\}\}.$$ Clearly, $S^{J}\neq\emptyset$, since it contains $\{ 0\}$. By equation \eqref{eq: 2}, we have that:
$$J\cap\bigvee S^{J}=J\cap\bigvee\{ I\in \ida| J\cap I=\{0\}\}=\bigvee\{J\cap I|J\cap I=\{0\}\}=\bigvee\{0\}=\{0\}.$$ In other words, $\bigvee S^{J}$ is the greatest ideal $I$ in $\ida$ such that $J\cap I=\{0\}$, which means that it is the pseudocomplement of $J$.
\end{proof}
In what follows, if $I\in\ida$, we denote the pseudocomplement of $I$ by $I^{\ast}$.\\
Let $\alga$ be a \L ukasiewicz near semiring. For any $a\in A$, we indicate by $I(a)$ the \textit{principal ideal} generated by $a$, i.e. the least ideal of $\ida$ that contains $a$.\\

Our next task will be to provide a full description of the principal ideals of $\ida$, for any \L ukasiewicz near semiring $\alga$. As it was proved in \cite{BCL}, the variety of \L ukasiewicz near semirings is congruence-permutable, as witnessed by the Mal'cev term 
\begin{equation}\label{eq:cng prmtblt}
 p(x,y,z)=((x\x y\al)\al\x z\al)+((z\x y\al)\al\x x\al))\al.
\end{equation}
By Werner's theorem \cite{Werner}, every reflexive binary relation on $A$ having the \textit{substitution property} with respect to operations of $\alga$ is a congruence on $\alga$. In particular, for any pair $(a,b)\in A^{2}$, the least reflexive relation having the substitution property, say $R(a,b)$, is the principal congruence $\theta (a,b)$, generated by $a,b$ in $\alga$. 

\noindent By Theorem \ref{thm:cst-ideal}, it follows that the ideal which is the $0$-coset of $\theta (a,0)$ is the least ideal containing $a$, namely $I(a)$. Recall that by a unary polynomial $p(x)$ we mean a unary term-function $t^{\alga}(x,e_{1},e_{2},...,e_{n})$ where $t$ is a $n+1$-ary term and $e_{1},...,e_{n}\in A$. Now, as shown in \cite{Chajda91}, one has that $(c,d)\in R(a,b)$ if and only if there exists a unary polynomial $p(x)$ on $A$ such that $c=p(a)$ and $d=p(b)$. Hence, $b\in I(a)$ if and only if $(b,0)\in\theta (a,0)$. Upon denoting by $Pol_{1}(\alga)$ the set of all unary polynomials of $\alga$, it follows directly that:
\begin{theorem}\label{thm: charact prin ideal}
For any $a\in A$, $$I(a)=\{p(a)|p\in Pol_{1}(\alga)\ \mbox{with}\ p(0)=0\}.$$
\end{theorem}
It is easily noticed that when dealing with a \L ukasiewicz semiring $\alga$, since $+$ is idempotent and due to the associativity and commutativity of $\x$ (cf. \cite[Theorem 2]{BCL}), polynomials in one variable on $\alga$ must be necessarily of the form 
\begin{equation}\label{eqn:poly}
 p(x)=xb+c, \text{ for }b,c\in A.
\end{equation}
Now, according to the reasoning above, it is also required that $p(0)=0$. Therefore, $c$ in condition \eqref{eqn:poly} must be $0$, and then we directly infer that the description of principal ideals in a \L ukasiewicz semiring $\alga$ can be simplified as follows:
\begin{corollary} For any $a\in A$,
\[
I(a)=\{a\x c|c\in A\}.
\]

\end{corollary} 
\section{Central elements and decompositions}\label{sec:2}

The aim of this section is discussing the notion of centrality in the variety of $\iota$-near semirings. 
This discussion will be relevant for the structure theory of \L ukasiewicz near semirings, since it provides a rather neat description of principal ideals generated by central elements, as well as for the application that the description of central elements has in the proof of a Cantor-Bernstein type theorem that we will propose in section \ref{sec:3}.

This section is based on the ideas developed in \cite{Sal} and \cite{Ledda13} on the general theory of \textit{Church algebras}. 

The notion of Church algebra is based on the simple observation that many well-known algebras, including Heyting algebras, rings with unit and combinatory algebras, possess a ternary term operation $q$ and term definable nullary operations $0,\,1$, satisfying the equations: 
\[
 q(1,x,y)\approx x \text{ and } q(0,x,y)\approx y .
\]

The term operation $ q $ simulates the behaviour of the if-then-else connective and, surprisingly enough, these rather simple conditions determine quite strong algebraic properties. 

An algebra $\mathbf{A}$\ of type $\nu$\ is a \emph{Church algebra}\ if there
are term definable constants $0^{\mathbf{A}},1^{\mathbf{A}}\in A$\ and a term
operation $q^{\mathbf{A}}$ such that, for all $a,b\in A$, 
\[
q^{\mathbf{A}}\left(
1^{\mathbf{A}},a,b\right)  =a\text{ and } q^{\mathbf{A}}\left(  0^{\mathbf{A}%
},a,b\right)  =b.
\]
A variety $\mathcal{V}$\ of type $\nu$\ is a Church
variety\ if every member of $\mathcal{V}$\ is a Church algebra with respect to
the same term $q\left(  x,y,z\right)  $\ and the same constants $0,1$. 

Following the seminal work of D. Vaggione \cite{Vaggio}, we say that an element \textit{e} of a Church algebra \textbf{A} is \textit{central} if the congruences $ \theta(e,0),\theta (e,1)  $ form a pair of factor congruences on \textbf{A}. A central element is said to be \emph{nontrivial} if it differs from $0$ and $1$. We denote the set of central elements (the \textit{centre}) of $\mb{A}$ by $\mathrm {Ce}({A})$.

Setting
$$ x\wedge y=q(x,y,0),\;\; x\vee y= q(x,1,y)\;\; x^{*}=q(x,0,1) $$
we recall a general result for Church algebras:
\begin{theorem}\label{th: Boolean algebra of centrals}
\emph{\cite{Sal}} Let $ \mathbf{A} $ be a Church algebra. Then 
\[
\mathrm {Ce}(\mathbf{A})=\langle \mathrm {Ce}(A),\wedge,\vee,^{*},0,1\rangle
\]
is a Boolean algebra which is isomorphic to the Boolean algebra of factor congruences of $\mathbf{A}$.
\end{theorem} 


If $\mathbf{A}$ is a Church algebra of type $\nu$ and $e\in A$ is a central
element, then we define $\mathbf{A}_{e}=(A_{e},g_{e})_{g\in\nu}$ to be the
$\nu$-algebra defined as follows:

\begin{equation}\label{eq:opAe}
A_{e}=\{e\wedge b:b\in A\};\quad g_{e}(e\wedge\overline{b})=e\wedge
g(e\wedge\overline{b}),
\end{equation}
where $ \overline{b} $ denotes the $n$-tuple $ b_1,...,b_n $ and $e\wedge\overline{b} $ is an abbreviation for $ e\wedge b_1,...,e\wedge b_n $.

In any Church algebra, central elements are amenable of a neat description as follows:
\begin{theorem}\label{prop: description of central elements in Church varieties}
If $\mb{A}$ is a Church algebra of type $ \nu $ and $ e\in A $, the following conditions are equivalent: 
\begin{itemize}
\item[(1)] e is central;
\item[(2)] for all $ a,b,\in A$ and for all $\overline{a}, \overline{b}\in A^{n}$:
\begin{itemize}
\item[\textbf{a)}] $ q(e,a,a)=a $,
\item[\textbf{b)}] $ q(e,q(e,a,b),c)=q(e,a,c)=q(e,a,q(e,b,c)) $,
\item[\textbf{c)}] $ q(e,f( \overline{a}),f( \overline{b}))= f(q(e,a_{1},b_{1}),...,q(e,a_{n},b_{n})) $, for every $ f\in\nu $, 
\item[\textbf{d)}] $ q(e,1,0)=e $.
\end{itemize}
\end{itemize}
\end{theorem}
By \cite[Theorem 4]{Ledda13}, we obtain the following theorem:

\begin{theorem}\label{th: decomposizione Church algebras}
\label{relat}Let $\mathbf{A}$ be a Church algebra of type $\nu$ and $e$ be a
central element. Then:

\begin{enumerate}
\item For every $n$-ary $g\in\nu$ and every sequence of elements $\overline
{b}\in A^{n}$, $e\wedge g(\overline{b})=e\wedge g(e\wedge\overline{b})$, so
that the function $h_{e}:A\rightarrow A_{e}$, defined by $h_{e}(b)=e\wedge b $, is a
homomorphism from $\mathbf{A}$ onto $\mathbf{A}_{e}$.

\item $\mathbf{A}_{e}$ is isomorphic to $\mathbf{A}/\theta(e,1)$. It follows
that $\mathbf{A}\cong\mathbf{A}_{e}\times\mathbf{A}_{e^{*}}$ for every
central element $e$, as in the Boolean case, under the mapping $f(a)\mapsto (h_{e} (a),h_{e^{\ast}} (a))$.
\end{enumerate}
\end{theorem}

This facts will be expedient in the context of $\iota$-near semirings. Indeed, they are a Church variety \cite[Definition 3.1]{Sal}.
\begin{lemma}\label{prop:Church variety}
The class of $\iota$-near semirings is a Church variety, with witness term 
\[
q(x,y,z) = (x\cdot y) + (x\al\cdot z).
\]
\end{lemma}
\begin{proof}
Direct computation.
\end{proof}
A straightforward interpretation of items a)--d) of Theorem \ref{prop: description of central elements in Church varieties} in our framework immediately provides that, given a $\iota$-near semiring $\alga$, the operations $ \wedge,\vee,^{*} $ in the Boolean algebra $\mathrm{Ce}(\alga)$ coincide with $ \cdot, + , \al$, respectively (cf. \cite[Proposition 3]{BCL}). 
\begin{lemma}\label{lem: cntrl-nvltv}
If $e$ is central in a $\iota$-near semiring $\alga$, and $a,b\in A$, then,
\begin{enumerate}
\item $e\cdot e=e$ (idempotency);
\item $e\cdot a=a \cdot e$ (commutativity);
\item $(e\cdot a)\cdot b= a\cdot (e\cdot b)$ (associativity).
\end{enumerate}
\end{lemma}
\begin{proof} In this proof we will freely use Theorem \ref{prop: description of central elements in Church varieties} and Lemma \ref{prop:Church variety}.\\
(1) $e=q(e,1,0)= q(e,1\x1,0\x0)=q(e,1,0)\x q(e,1,0)=e\x e$.\\
(2) $e\x a=q(e,1,0)\x q(e,a,a)= q(e,1\x a,0\x a)=q(e,a\x 1,a\x 0)=q(e,a, a)\x q(e,1,0)=a\x e$.\\
(3) $(e\x a)\x b=q(e, a,0)\x q(e,b,b)=q(e,a\x b,0)=e\x(a\x b)=q(e,a,a)\x q(e,b,0)=a\x(e\x b)$.
\end{proof}

In \L ukasiewicz near semirings conditions (a)-(d) in Theorem \ref{prop: description of central elements in Church varieties} translate as follows:
a) is trivially satisfied, by Lemma \ref{lem: cntrl-nvltv}.
For condition b),
\[
(e\x c)+(e\al\x((e\x b)+(e\al \x a)))=(e\x c)+(e\al\x a);
\]
and 
\[
(e\x c)+(e\al\x a)=(e\x((e\x c)+(e\al\x b)))+(e\al\x a).
\]
As regards condition c), if $f$ is the constant $0$ or $1$, then clearly $(e\x 1)+(e\al\x 1)=1$, and $(e\x 0)+(e\al\x0)=0$.
If $f$ is $+$, 
\[
(e\x(b_{1}+b_{2}))\x(e\al\x(a_{1}+a_{2}))=((e\x b_{1})+(e\al\x a_{1}))+((e\x b_{2})+(e\al\x a_{2})).
\]
In case $f$ is $\x$, 
\[
(e\x(b_{1}\x b_{2}))+(e\al\x(a_{1}\x a_{2}))=((e\x b_{1})+(e\al\x a_{1}))\x((e\x b_{2})+(e\al\x a_{2})).
\]
In case $f$ is $\al$, 
\[
(e\x a\al)+(e\al\x b\al)=((e\x a)+(e\al\x b))\al.
\]

Finally, condition d) is obviously satisfied by any element in a \L ukasiewicz near semirings.

As we have already seen in section \ref{sec:1}, Theorem \ref{thm: charact prin ideal} provides a full description of principal ideals generated by elements of a \L ukasiewicz near semiring. Moreover, generalizing the Boolean case, central elements produce a direct decomposition of these algebras. Due to this fact,  in what follows we will see that the ideals generated by central elements can be described easily.
\begin{definition}
Let $\Lns$ be a \L ukasiewicz near semiring and $I,J\in\ida$. $I,J$ form a pair of \emph{factor ideals} if and only if 
$$I\cap J=\{0\}\quad\mbox{and}\quad I\vee J=A.$$ 
\end{definition}
By the fact that $\ida$ and $\Con$ are the universes of two isomorphic algebraic distributive lattices, it is direct to verify that $I,J$ form a pair of factor ideals if and only if $\theta (I)$ and $\theta (J)$ form a pair of factor congruences.\\
Upon recalling that, for $a\in A$, the interval $[0,a]$ corresponds to the set $\{x|x\leq a\}$, from the last notion introduced, the following theorem is obtained.
\begin{theorem}\label{thm: char idl prin centr}
Let $e$ be a central element of a \L ukasiewicz near semiring $\Lns$. Then $I(e)=[0,e]$.
\end{theorem}
\begin{proof}
By Theorem \ref{th: decomposizione Church algebras}, $e$ is central if and only if, for any $a\in A$, the mapping $f: A\rightarrow [0,e]\times [0,e\al]$, defined by $f(a)\mapsto (h_{e} (a),h_{e\al} (a))$, is a direct decomposition of $\mb A$. Let $\theta_{1}$ and $\theta_{2}$ be the factor congruences associated to $\ker(\pi_{2}\circ f)$ and $\ker(\pi_{1}\circ f)$, respectively, where $\pi_{i}$ ($i\in\{1,2\}$) is the natural projection map. We denote by $I_{i}$ ($i=1,2$) these kernels. Then, $e\in I_{1}$ and $e\al\in I_{2}$. Thus, $I(e)\subseteq I_{1} =[0,e]$ and $I(e\al)\subseteq I_{2}=[0,e\al]$. Hence, $I(e)\cap I(e\al)=\{ 0\}$. It is clear that, for a central element $e$, one has that $1=e+e\al\in I(e)\vee I(e\al)$. So $I(e)\lor I(e\al)= A$. Hence, $I(e)$ and $I(e\al)$ form a pair of factor ideals with $I(e)\subseteq I_{1}$, $I(e\al)\subseteq I_{2}$. Since $I_{1}$ and $I_{2}$ are factor ideals, we have that
	$$I(e)=I_{1} =[0,e].$$
\end{proof}
{A few basic results about the pseudocomplements are subsumed in the following lemma.}
\begin{lemma}\label{prop: annhil ideal}
Let $\Lns$ be a \L ukasiewicz near semiring, $I,J\in\ida$ and $a\in A$. Then:
\begin{enumerate}
\item $I\subseteq {I\cmp}\cmp$;
\item If $I\subseteq J$, then $J\cmp\subseteq I\cmp$;
\item $I\cmp={{I\cmp}\cmp}\cmp$;
\item $(I, I\cmp)$ is a pair of factor ideals if and only if $I\vee I\cmp=A$;
\item $(I(a),I(a\al))$ is a pair of factor ideals if and only if 
\[
I(a\al)=I(a)\cmp\mbox{ and } I(a)\vee I(a\al)=A.
\]
\end{enumerate}
\end{lemma}
\begin{proof} (1), (2), (3) and (4) are straightforward. For (5), if $(I(a),I(a\al))$ is a pair of factor ideals, obviously $I(a)\vee I(a\al)=A$. Now, clearly $I(a\al)\subseteq I(a)\cmp$. Furthermore, if $b\in I(a)\cmp$,  then $b=0$ or $b\notin I(a)$ and $b\in I(a\al)$ for $I(a)\vee I(a\al)=A$. In any case one has  $b\in I(a\al)$. The converse follows immediately.
\end{proof}

Let us now briefly elaborate on the previous results.
Consider a \L ukasiewicz near semiring $\Lns$, and let $\skel$ be the \textit{skeleton} of $\ida$, namely 
\[
\skel =\{ I^{\ast}|I\in\ida\}.
\]
 
By a theorem due to V. Glivenko, later proved in its full generality by O. Frink (see e.g. \cite{Gratzer:2011}), since (by Theorems \ref{thm: isom ida con} and \ref{thm: ida pseudocompl}) $\ida$ is an algebraic pseudocomplemented lattice, it turns out that $\skel$ is a Boolean lattice bounded by the trivial ideals $\{0\}$ and $A$. With a slight abuse of language, we may identify the skeleton with the Boolean algebra $\skel= \langle\skel ,\wedge,\vee,^{\ast},\{0\},A\rangle$ where, for any $I,J\in\skel$, $\wedge$ is $\cap$, $\vee$ is defined by $I\vee J=(I^{\ast}\wedge J^{\ast})^{\ast}$. Trivially, for any $I\in\skel$, $I$ and $I^{\ast}$ form a pair of complementary factor ideals. 

Now, by Theorem \ref{thm:idl-cngrnc}, if $I,I^{\ast}\in\skel$, then $I=[0]_{\theta(I)}$ and $I^{\ast}=[0]_{\theta(I^{\ast})}$. By $I\vee I^{\ast}=A$ one obviously has that $\theta(I)\vee\theta(I^{\ast})=\nabla$ and $I\wedge I^{\ast}=\{0\}$ implies that $\theta(I)\wedge\theta(I^{\ast})=\Delta$.

Conversely, if $(\theta,\theta')$ is a pair of complementary factor congruences, then their $0$-cosets, say $I$ and $J$, respectively, form a pair of complementary factor ideals. Indeed, it is easily seen that $I=J^{\ast}$ and $J=I^{\ast}$. 

In fact, by Lemma \ref{prop: annhil ideal}, one has that $I\subseteq I^{\ast\ast}$ and $J\subseteq J^{\ast\ast}$. Hence, $I^{\ast\ast}\vee J^{\ast\ast}=A$. Moreover, since $I$ and $J$ form a pair of complementary factor ideals, one has $I\subseteq J^{\ast}$ and $J\subseteq I^{\ast}$. Thus, by (2) of Lemma \ref{prop: annhil ideal}, $I^{\ast\ast}\subseteq J^{\ast}$ implies that if $x\in I^{\ast\ast}$ and $x\neq0$, then $x\in J^{\ast}$ and $x\notin J^{\ast\ast}$. Hence, one has that $I^{\ast\ast}\wedge J^{\ast\ast}=\{0\}$ and, by the unicity of complements in $\skel$, we can conclude that $I^{\ast\ast}=J^{\ast}$ and $J^{\ast\ast}=I^{\ast}$. In fact, suppose \emph{ex absurdo} that $x\notin J$ and $x\in J^{\ast\ast}$. Hence, $x\notin J^{\ast}=I^{\ast\ast}$ and $x\notin I$. So $I\subseteq J$ and since $I\cap J=\{0\}$ this is a contradiction. 
This implies that $J=J^{\ast\ast}=I^{\ast}$. Similarly, $J^{\ast}=I$. Then, we can conclude that there is a one-to-one correspondence between pairs of complementary factor ideals in $\skel$ and pairs of complementary factor congruences in $\Con$. 

Since $\alga$ is congruence-distributive (see \cite{BCL}) one has that the sublattice of $\Con$ that contains all pairs of complementary factor congruences on $\alga$, that we denote by $\Con_{\mathrm{F}}$, is Boolean. Exploiting the same mapping $f$ of Theorem \ref{thm: isom ida con}, one can easily observe that $\skel\cong\Con_{\mathrm{F}}$. 

Finally, by Theorem 3.7 in \cite{Sal} one has that $\Con_{\mathrm{F}}\cong\mathrm {Ce}(\mathbf{A})$, where $\mathrm {Ce}(\mathbf{A})$ is the Boolean lattice of central elements of $\alga$. In particular, it shows that the map $e\mapsto\theta(e,0)$ is a bijective correspondence between $\mathrm{Ce}(\alga)$ and $\Con_{\mathrm{F}}$.
Moreover, for any $e,d\in\mathrm{Ce}(\alga)$, the elements $e\al$, $e\wedge d$, $e\vee d$ are central and naturally associated with the factor congruences $\theta(e,1)=\theta(e\al,0)$,  $\theta(e,0)\cap\theta(d,0)$ and $\theta(e,0)\vee\theta(d,0)$, respectively. Hence, for any pair of complementary factor congruences $(\theta,\theta')$ one has that $(\theta,\theta')=(\theta(e,0),\theta(e\al,0))$, for some $e\in\mathrm{Ce}(\alga)$. 

Summarizing the observations above, we have that $\skel$ coincides with the Boolean lattice of pairs of complementary factor ideals $(I,I^{\ast})$ in $\ida^{2}$, which are nothing but $0$-cosets of pairs of complementary factor congruences of the form ($\theta(e,0),\theta(e\al,0)$), for an element $e$ in $\mathrm{Ce}(\alga)$. Thus, it directly follows that $$\skel =\{I(e)|e\in\mathrm {Ce}(\mathbf{A})\}$$ and by Theorem \ref{thm: char idl prin centr}, $I^{\ast}=[0,e]$, for some $e\in\mathrm{Ce}(\alga)$. 

\section{A Cantor-Bernstein-type Theorem for $\iota$-near semirings}\label{sec:3}
We close this article with an application of the theory of central elements in $\iota$-near semirings. Namely, we propose a version of the Cantor-Bernstein Theorem for join $\sigma$-complete $\iota$-near semirings, with $\sigma$-complete algebras of central elements. More specifically, for a $\iota$-near semiring $\alga$, if $\{a_{i}\}_{i\in I}$, such that $|I|\leq \sigma$, then $\bigvee_{i\in I}a_{i}$ exists, and $\mathrm{Ce}(\alga)$ is a $\sigma$-complete Boolean algebra. 
This result was first shown in \cite{Sikorski} (see also \cite{Tarski}) for Boolean algebras and subsequently extended to MV-algebras (with Boolean elements), orthomodular lattices, and other classes of algebras enjoying suitable properties, such as having an underlying lattice structure (see \cite{DeSimone}, \cite{Freytes}). Since \L ukasiewicz near semirings generalize the notion of MV-algebra, it is natural to wonder whether a version of the Cantor-Bernstein Theorem could be widened for weaker structures, like $\iota$-near semirings.

Upon recalling that central elements commute with any other element (cf. Lemma \ref{lem: cntrl-nvltv}), in order to prove the main result of this section, we start with the following
\begin{lemma}\label{prprts-CB}
Let $\alga$ be a $\iota$-near semiring, and $e\in\mathrm{Ce}(\alga)$, and $a,b\in A$.
\begin{enumerate}
\item if $a\leq e$, then $ae=a$;
\item $eb=e\land b$;
\item if $\{a_{i}\}_{i\in N}\subseteq A$, then 
\[
e\land(\Sigma_{n\in N}a_{n})=\Sigma_{n\in N}(e\land a_{n}).
\]
\end{enumerate}
\end{lemma}
\begin{proof}
%

(1) If $a\leq e$, then $ae\al\leq ee\al=0$, because $e$ is central. Therefore, $ea=ea+0=ea+e\al a$, because $e\al $ commutes, and so $ ae=ea=(e+e\al)a=1a=a$, because $e+e\al=e\lor e\al=1$. 

(2) First, observe that $be\leq b,e$, by \cite[Lemma 1]{BCL}. If $a\leq e, b$, then $a+e=e$ and $b+a=b$. Note that $eb+a=eb+ea$, by the previous item, and $eb+ea=e(b+a)=eb$, because $e$ is central and therefore commutes. Therefore, $a\leq eb=e\land b$. 

(3) By induction on $n$. If $n=0$, the claim is obvious. Suppose that the statement is true for $n-1$. Then $(\Sigma_{i=1}^{n}a_{i})\land e=(\Sigma_{i=1}^{n}a_{i})e=(\Sigma_{i=1}^{n-1}a_{i}+a_{n})e=\Sigma_{i=1}^{n-1}a_{i}e+a_{n}e=\Sigma_{i=1}^{n}a_{i}e=\Sigma_{n\in N}(e\land a_{n})$.
\end{proof}
When there is no confusion possible, we will use $\cdot$ and $\land$ ($+$ and $\lor$) as synonyms, respectively.  
The following lemma completes the results of the previous lemma.
\begin{lemma}\label{lemm: hmmrpsms cntrl}
Let $\alga,\algb$ be a join $\sigma$-complete near semirings and $\gamma:\alga\to\algb$ an isomorphism. Then,
\begin{enumerate}
\item if $a\in \mathrm{Ce}(\alga)$, then $\gamma(a)\in \mathrm{Ce}(\algb)$;
\item if $a\in \mathrm{Ce}(\alga)$, then $\gamma\upharpoonright[0,a]$ is isomorphic to $[0,\gamma(a)]$;
\item if $\{a_{i}\}_{i\in N}\subseteq \mathrm{Ce}(\alga)$, $\bigvee_{n\in N}a_{n}=1$, and for $i\neq j$ $a_{i}\land a_{j}=0$, then $\mathbf A$ is isomorphic to $\Pi_{n\in N}[0,a_{n}]$.
\end{enumerate}
\end{lemma}
\begin{proof}
(1) and (2) are straightforward. (3) let $\{a_{i}\}_{i\in N}$ be a family of central elements with the required properties. Let us call $\beta $ the map from $\alga $ to $\Pi_{n\in N}[0,a_{n}]$ defined, for $a\in A$, by $a\mapsto (a\land e_{i}:i\in N)$. Clearly, if $i\neq j$, then, in case $b\leq a_{i},a_{j}$, we have that $b\leq a_{i}\land a_{j}=0$. Thus, $[0,a_{i}]\cap [0,a_{j}]=\{0\}$, which implies injectivity. Clearly, $\beta(1)=\beta(\bigvee_{n\in N}a_{n})=(a_i:i\in N)$. Let $(b_i:i\in N)$. Then, $\beta(\bigvee_{i\in N}b_{i})=((\bigvee_{i\in N}b_{i})\land a_{i}:i\in N)=(\bigvee_{i\in N}b_{i}\land a_{i}:i\in N)=(b_{i}:i\in N)$. The fact that $\beta$ preserves the operations directly follows from general results on central elements in a Church algebra \cite{Sal}.
\end{proof}
We now have all the elements required for proving our main theorem. Recall that, given a near semiring $\alga$ and $a\in\mathrm{Ce}(\alga)$, the interval $[0,a]$ is an algebra whose operations are the same as in $\alga$ although adequately ``constrained'' to the considered subset of $A$ (see Theorem \ref{th: decomposizione Church algebras}).
\begin{theorem}\label{thm:cantorbernstein}
	Let $\Lns$ and $\Lnsb$ be join $\sigma$-complete $\iota$-near semirings, such that $\mathrm{Ce}(\alga)$ and $\mathrm{Ce}{(\algb)}$ are $\sigma$-complete Boolean algebras. If $\alga\cong [0,b]$ and $\algb\cong [0,a]$ with $b\in\mathrm{Ce}(\algb)$ and $a\in\mathrm{Ce}{(\alga)}$, then $\alga\cong\algb$.  
\end{theorem}
\begin{proof}
Let $\gamma:\alga\rightarrow [0,b]$ and $\beta:\algb\rightarrow [0,a]$ be isomorphisms with $a\in\mathrm{Ce}(\alga)$ and $b\in\mathrm{Ce}(\algb)$. Without loss of generality, we can safely assume that $0<a,b<1$. We recursively define, as in the proof of \cite[Theorem 4.1]{DeSimone}, the following pair of infinite sequences:
\begin{center}
$\begin{array}{lr}
v_{0}=1\quad\quad\quad\quad\quad\quad u_{0}=1\\
v_{n+1}=\beta(u_{n})\quad\quad\quad u_{n+1}=\gamma(v_{n}).
\end{array}$
\end{center}
Since $1\in \mathrm{Ce}(\alga)\cap\mathrm{Ce}(\algb)$ and $\gamma,\beta$ are isomorphisms, one has, by Lemma \ref{lemm: hmmrpsms cntrl}(1), that $u_{n}\in\mathrm{Ce}(\algb)$ and $v_{n}\in\mathrm{Ce}(\alga)$ for any $n\in N$. Indeed, by induction on $n$, we obtain that $v_{n}=\beta(u_{n-1})=\beta(\gamma(v_{n-2}))$. Since $\beta\circ\gamma$ is still an isomorphism, a straightforward application of the induction hypothesis yields $v_{n}\in \mathrm{Ce}(\alga)$. Similarly, $u_{n}\in\mathrm{Ce}(\algb)$. Furthermore, it can be seen that $$v_{0}>v_{1}>...>...\quad \mathrm{and}\quad   u_{0}>u_{1}>...>...\ .$$ 

In fact, by induction on $n$, one has that $v_{0}=1+v_{1}=1$ (since any $\iota$-near semiring is integral). Hence, $v_{0}\geq v_{1}$. Now, suppose that $v_{k}+v_{k+1}=v_{k}$ for any $k<n$. It can be seen that $v_{n}+v_{n+1}=\beta(u_{n-1})+\beta(u_{n})=(\beta\circ\gamma)(v_{n-2}+v_{n-1})=(\beta\circ\gamma)(v_{n-2})=v_{n}$. Similarly, $u_{n}+u_{n+1}=u_{n}$, for any $n\in N$. Clearly, $v_{n}\gneq v_{n+1}$ ($u_{n}\gneq u_{n+1}$) follows from the injectivity of $\beta$ ($\gamma$).

Indeed, since $\mathrm{Ce}(\alga)$ and $\mathrm{Ce}(\algb)$ are $\sigma$-complete Boolean algebras (see Theorem \ref{th: Boolean algebra of centrals}), we can define the following $$v_{\infty}=\bigwedge_{n\in N} v_{n}\quad\quad\mathrm{and}\quad\quad u_{\infty}=\bigwedge_{n\in N} u_{n}.$$ Recall that, by Lemma \ref{prprts-CB}(2), since all $v_{n}$, $n\in N$, are central, we obtain that $\bigwedge_{n\in N}v_{n}=\prod_{n\in N}v_{n}$. 
Similarly for $u_{n}$, $n\in N$. 

Moreover, a simple computation proves that $\gamma(v_{\infty})=\gamma(\bigwedge_{n\in N}v_{n})=\bigwedge_{n\in N}\gamma(v_{n})=\bigwedge_{n\in N}u_{n+1}=u_{\infty}$ as well as $\beta(u_{\infty})=v_{\infty}$. We define the following $$e_{n}=v_{n}\wedge v_{n+1}\al\quad\mathrm{and}\quad d_{n}=u_{n}\wedge u_{n+1}\al.$$ Let us note that $\gamma(e_{n})=\gamma(v_{n}\wedge v_{n+1}\al)=\gamma(v_{n})\wedge\gamma(v_{n+1})\al=u_{n+1}\wedge u_{n+2}\al=d_{n+1}$. Similarly, $\beta(d_{n})=e_{n+1}$. Now, it is easily seen that $e_{n-1}=v_{n}\al$ and $d_{n-1}=u_{n}\al$ for any $n\in N^{+}$. Indeed, since the latter case can be handled similarly, we prove the former. We have that $e_{0}=v_{0}\wedge v_{1}\al=1\x v_{1}\al=v_{1}\al$. Suppose that $e_{k-1}=v_{k}\al$ for any $k<n$. We obtain that 
$e_{n-1}=v_{n-1}\wedge v_{n}\al=(v_{n-1}\al)\al\wedge v_{n}\al=(v_{n-2}\x v_{n-1}\al)\al\x v_{n}\al=(v_{n-2}\al + v_{n-1})\x v_{n}\al$, by centrality and De Morgan laws, and then $v_{n-2}\al\x v_{n}\al +v_{n-1}\x v_{n}\al=v_{n}\al +v_{n-1}\x v_{n}\al=(v_{n-1}+1)\x v_{n}\al =v_{n}\al.$ Hence:
$$\bigvee_{n\in N^{+}}e_{n-1}=\bigvee_{n\in N^{+}}v_{n}\al=(\bigwedge_{n\in N^{+}}v_{n})\al=v_{\infty}\al$$ as well as $$\bigvee_{n\in N^{+}}d_{n-1}=u_{\infty}\al.$$ Thus, we have that $v_{\infty}\vee (\bigvee_{n\in N}e_{n})=1$ and $u_{\infty}\vee(\bigvee_{n\in N}d_{n})=1$. Furthermore, let us note that $e_{m}\wedge e_{n}=0$ and $d_{m}\wedge d_{n}=0$ for any $n\neq m$. In fact, suppose without loss of generality that $m>n$. It can be verified that $e_{m}\wedge e_{n}=v_{m}\wedge v_{m+1}\al\wedge v_{n}\wedge v_{n+1}\al=(v_{m}\wedge v_{n+1})\wedge v_{m+1}\al\wedge v_{n}\wedge v_{n+1}\al=0$. The latter case can be handled similarly. Moreover, a little thought shows that $v_{\infty}\wedge e_{n}=0$ as well as $u_{\infty}\wedge d_{n}=0$, for any $n\in N$. 

Finally, a direct application of Lemma \ref{lemm: hmmrpsms cntrl} yields 

\[
\alga\cong [0,v_{\infty}]\times [0,e_{0}]\times [0, e_{1}]\times\dots\times\dots
\]

and 
\[
\algb\cong [0,u_{\infty}]\times [0,d_{0}]\times [0,d_{1}]\times\dots\times\dots\ .
\]
 Recall that $\gamma(v_{\infty})= u_{\infty}$ and $\gamma(e_{n})=d_{n+1}$ as well as $\beta(d_{n})=e_{n+1}$, for any $n\in N$. Hence, by Lemma \ref{lemm: hmmrpsms cntrl}, we obtain that $[0,e_{\infty}]\cong [0,\gamma(e_{\infty})]=[0,d_{\infty}]$, $[0, e_{n}]\cong [0,\gamma(e_{n})]=[0,d_{n+1}]$ and  $[0,d_{n}]\cong [0,\beta(d_{n})]=[0,e_{n+1}]$. Thus, in general, we have that $\alga\cong\algb$.
\end{proof}
\quad\\
By virtue of Definition \ref{def: Lnrsmrng}, Theorem \ref{thm: char idl prin centr} and the definition of $\alga_{e}$, with $e\in \mathrm{Ce}(\alga)$ (see \ref{eq:opAe}), if $\alga$ and $\algb$ are \L ukasiewicz near semirings, then they can be regarded as \textit{trivial} principal ideals $[0,1]$ (generated by $1\in \mathrm{Ce}(\alga)\cap\mathrm{Ce}(\algb)$) of $\ida$ and $\idb$, respectively, Hence, we conclude the following:

\begin{corollary} Let $\alga$ and $\algb$ be join $\sigma$-complete \L ukasiewicz near semirings such that $\mathrm{Ce}(\alga)$ and $\mathrm{Ce}(\algb)$ are $\sigma$-complete Boolean algebras. Then, $\alga\cong\algb$ if and only if there are central elements $a\in\mathrm{Ce}(\alga)$ and $b\in\mathrm{Ce}(\algb)$ such that $\alga\cong I(b)$ and $\algb\cong I(a)$. 
\end{corollary}
\begin{proof}
Suppose that there are central elements $a\in\mathrm{Ce}(\alga)$ and $b\in\mathrm{Ce}(\algb)$ such that $\algb\cong I(a)$ and $\alga\cong I(b)$, respectively. By Theorem \ref{thm: char idl prin centr}, it follows that $I(a)=[0, a]$ as well as $I(b)=[0,b]$. Hence, Theorem \ref{thm:cantorbernstein} ensures that $\alga\cong \algb$. Conversely, if $\alga\cong \algb$, then, upon noticing that $1$ is the greatest element in any \L ukasiewicz near semiring, and it is also central, we have that $\alga\cong [0, 1^{\algb}]=\algb$ and viceversa.
\end{proof}


\subsection*{Aknowledgements}
The research of the I. Chajda is supported by Project I 1923-N25 by Austrian Sci. Fund. (FWF) and Czech Grant Agency (GA\v CR), project 15-34697L . D. Fazio and A. Ledda gratefully acknowledge the support of the Horizon 2020 program of the European Commission: SYSMICS project, Proposal Number: 689176, MSCA-RISE-2015, the support of the Italian Ministry of Scientific Research (MIUR) within the FIRB project ``Structures and Dynamics of Knowledge and Cognition'', Cagliari: F21J12000140001. A. Ledda expresses his gratitude for the support of Fondazione Banco di Sardegna within the project ``Science and its Logics: The RepresentationÕs Dilemma'', Cagliari, Project Number: F72F16003220002. Finally, we express our gratitude to the anonymous reviewers for their valuable suggestions.

\end{document}